\documentclass{article}

\author{Fabio Gobbi \hspace{0.4cm} Sabrina Mulinacci\\
\vspace{0.4cm} University of Bologna - Department of Statistics}

\usepackage{amsmath}
\usepackage{amsthm}
\usepackage{amsfonts}
\usepackage{amssymb}
\usepackage{latexsym}
\usepackage[dvips]{graphicx}
\usepackage{pstricks}
\usepackage{epsfig}
\usepackage{a4}
 \oddsidemargin=\evensidemargin
\newtheorem{theorem}{Theorem}[section]

\newtheorem{proposition}{Proposition}[section]
\newtheorem{definition}[theorem]{Definition}

\begin{document}

\title{Gaussian autoregressive process with dependent innovations. Some asymptotic results.}

\maketitle

\begin{abstract}
In this paper we introduce a modified version of a gaussian
standard first-order autoregressive process where we allow for a
dependence structure between the state variable $Y_{t-1}$ and the
next innovation $\xi_t$. We call this
model dependent innovations gaussian AR(1) process
(DIG-AR(1)). We analyze the moment and temporal dependence properties
of the new model. After proving that the OLS estimator does not
consistently estimate the autoregressive parameter, we introduce an
infeasible estimator and we provide its $\sqrt{T}$-asymptotic
normality.

 \end{abstract}
\bigskip

  \vspace{0.2cm}


\medskip

\section{Introduction}

In this paper we consider a modified version of the standard first-order
autoregressive process, $Y_t=\phi Y_{t-1} + \xi_t$, in which, unlike the classical case, we assume
that the state variable at the time $t-1$, $Y_{t-1}$, and the next
innovation $\xi_t$ are no longer independent.

The model here considered is a particular case of the class of the C-convolution based processes studied in
Cherubini et al. (2011) and Cherubini et al. (2012). Recent literature has mainly focused on stationary copula-based Markov process
(see, among the others, Chen and Fan, 2006, and Beare, 2010), while the mentioned C-convolution based processes have a stationary dependence structure between
each level and the next innovation generating a process which no longer stationary. 

In particular, this paper focuses on the 
case where the pairs $(Y_{t-1},\xi_t)_t$ have a joint
distribution function given by a gaussian copula with a
time-invariant parameter $\rho$ and the innovations $(\xi_t)_t$ are
identically distributed with a gaussian distribution: in addition, we assume that the resulting stochastic process is a Markov process. 
This model first appeared in
Cherubini et al. (2016) where some moment properties have been
analyzed without, however, considering temporal dependencies. The unit root case ($\phi=1$) is studied in 
Gobbi and Mulinacci (2017) while here we will concentrate on the case $\vert \phi\vert <1$.
We call the model
``dependent innovations gaussian AR(1) process (DIG-AR(1))''. 

The analysis of the model starts from the study of the implied temporal
dependence properties both in the sequence $(Y_t)_t$ and in the
sequence of innovations $(\xi_t)_t$ which, as expected, are no
more independent. As a consequence, statistical inference on the
model changes significantly. In particular, estimating the
autoregressive parameter can no longer be made by ordinary least
squares (OLS). As we shall prove, the OLS
estimator is not consistent for $\phi$ and, as expected, the
asymptotic bias depends on $\rho$. To overcome this drawback, in
this paper we propose a new infeasible estimator of $\phi$ which
allows us to achieve a $\sqrt{T}$-consistency and asymptotic
normality. The new estimator takes into account a correction due
to the presence of $\rho$ whose importance will be emphasized both
as regards the estimation procedure and as regards the dynamics of
the process.

The paper is organized as follows. Section \ref{sec2} introduces the
DIG-AR(1) process. Section \ref{sec3} shows the autocorrelation functions
of $(Y_t)_t$ and $(\xi_{t})_t$. In Section \ref{sec4} we prove that the OLS
estimator is not consistent for the autoregressive parameter and we introduce a new infeasible estimator 
which is proved to be $\sqrt T$-asymptotically normal. Section \ref{sec5} concludes.

\section{The model}\label{sec2}
In this paper we consider a generalized version of the standard
stationary gaussian first-order autoregressive process, AR(1),
defined as $$Y_t=\phi Y_{t-1}+\xi_t,$$ where $Y_0=0$ a.s.,
$|\phi|<1$ (condition for the weak stationarity) and the sequence
of innovations $(\xi)_{t\geq 1}$ is i.i.d. and normally
distributed with zero mean and standard deviation $\sigma_{\xi}$
(gaussian white noise process): as a consequence, $\xi_t$ is
independent of $Y_{t-1}$ for all $t$. It is just the case to
recall that in this framework we have
$$\left\{%
\begin{array}{ll}
    \mathbb{E}[Y_t]=0, \ t\geq 1 \\
    \mathbb{E}[Y_t^{2}]=\frac{\sigma^2_{\xi}}{1-\phi^2}, \ t \geq 1  \\
    \eta_k=corr(Y_t,Y_{t+k})=\phi^k, \ k=1,2,...\\
\end{array}%
\right.
$$
For a detailed discussion of AR(1) processes see, among others,
Hamilton (1994) and Brockwell and Davis (1991). The modified
version of the AR(1) process considered in this paper is inspired
by the Markov processes modeling with dependent increments
considered in Cherubini et al. (2011, 2012, 2016): the idea is to
relax the assumption of independence between $\xi_t$ and
$Y_{t-1}$, allowing for some time-invariant gaussian dependence.

More precisely
\begin{definition}\label{DIG-AR}
A dependent innovations gaussian AR(1) process (DIG-AR(1)) is a discrete time stochastic process defined as
\begin{equation}\label{C_AR(1)}Y_t=\phi Y_{t-1}+\xi_t,\, t\in\mathbb N\end{equation}
where $Y_0=0$ a.s., $|\phi|<1$, $(\xi_t)_t$ are identically distributed with
$\xi_t\sim N(0,\sigma^2_{\xi})$ and the copula associated to the
vector $(\xi_t,Y_{t-1})$ is gaussian with time-invariant parameter $\rho$ with
$|\rho|<1$ for all $t$.
\end{definition}
\smallskip

Therefore, in this version of the model the sequence of
innovations is no more a gaussian white noise but it has a
temporal dependence structure that we will analyze in the
following. Moreover, in Section 4.3.1 of Cherubini et al.
(2016), the authors determine the variance of $Y_t$ and its
asymptotic convergence, that is
$$V_t^{2}=\sigma_{\xi}^2\left(\phi^{2(t-1)}+\sum_{i=1}^{t-1}\phi^{2(i-1)}\right)+2\rho \sigma_{\xi}
\sum_{i=1}^{t-1}\phi^{2i-1}V_{t-i},$$ and
$$V_t \underset{t\uparrow +\infty}\longrightarrow \frac{\sigma_{\xi}\left(\rho \phi +\sqrt{\rho^2 \phi^2+1-\phi^2}\right)}{1-\phi^2}.$$
The limit, that we denote with
$\bar{V}(\rho,\phi)$, can be written as
$$\bar{V}(\rho,\phi)=
S \frac{\left(\rho \phi +\sqrt{\rho^2
\phi^2+1-\phi^2}\right)}{\sqrt{1-\phi^2}},$$ where
$S=\frac{\sigma_{\xi}}{\sqrt{1-\phi^2}}$ is the standard deviation of the AR(1) process. It is just the case to
observe that, when $\rho=0$, $\bar{V}(0,\phi)=S$. As
shown in Figure 1, we can notice that $\bar{V}(\rho,\phi)>S$ if
$\rho \phi
>0$ whereas $\bar{V}(\rho,\phi)<S$ if $\rho \phi <0$. In plain
words, $Y_t$ becomes more and more volatile if the correlation
between $\xi_t$ and $Y_{t-1}$ has the same sign of the
autoregressive coefficient. On the other hand, the volatility
tends to decrease if they have opposite sign.

\begin{figure}[h]
\begin{center}
\includegraphics[width=10cm,height=8cm]{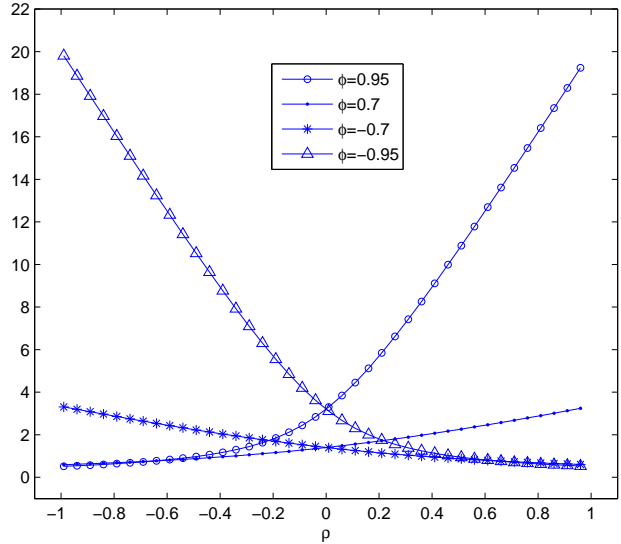}
\caption{Graph of $\bar{V}(\rho,\phi)$ as a function of $\rho$ for
different values of $\phi$. Here we fix
$\sigma_{\xi}=1$.}\label{fig1}
\end{center}
\end{figure}

In addition to the above assumptions, we assume that the stochastic process $(Y_t)_t$ is a Markov process.
As proved in Darsow et al.
(1992) a necessary condition for a process $(Y_t)_t$ to be Markovian is that, if $C_{s,t}$ is the copula associated to $(Y_s,Y_t)$, then
$$C_{s,t}(u,v)=C_{s,r}\ast C_{r,t}(u,v)=\int_0^1\frac{\partial}{\partial w}C_{s,r}(u,w)\frac{\partial}{\partial w}C_{r,t}(w,v)\, dw,\, \forall s<r<t.$$
Morover, as shown in section 3.2.3 in Cherubini et al. (2012), if copulas $C_{s,r}$ and $C_{r,t}$ are gaussian with parameters $\tau_{s,r}$
and $\tau_{r,t}$, respectively, then $C_{s,t}$ is again gaussian with parameter 
\begin{equation}\label{taustar}\tau_{s,t}=\tau_{s,r}\tau_{r,t}.\end{equation}

\section{Autocorrelation functions and mixing-properties}\label{sec3}

In this subsection we study the behavior of the autocorrelation
functions of the process $(Y_t)_t$ when $t \longrightarrow
+\infty$. It is just the case to recall that in the standard AR(1)
process the $k$-th order autocorrelation function of $(Y_t)_t$
depends only on $k$ and it is equal to $\phi^k$. In our more
general setting, this is no longer true. In section 4.3.1 of
Cherubini et al. (2016) it is shown that the copula between $Y_t$
and $Y_{t+1}$ is gaussian with time-dependent parameter given by
$$\tau_{t,t+1}=\frac{\phi V_t+\rho \sigma_{\xi}}{V_{t+1}},$$
and
\begin{equation}\label{tau1}\tau_{t,t+1}\underset{t\uparrow +\infty}\longrightarrow \phi+\frac{\rho
\sigma_{\xi}}{\bar{V}(\rho,\phi)}=\bar{\tau}(\rho,\phi).\end{equation}
\bigskip

The limit of the $k$-th order autocorrelation function of the
Markov process $(Y_t)_t$ is a function of $k$, $\sigma_\xi$ and $\rho$ as the
following proposition shows.

\begin{proposition} The $k$-th order autocorrelation function of the DIG-AR(1) Markov process $(Y_t)_t$ tends to $\left(\phi+\frac{\rho
\sigma_{\xi}}{\bar{V}(\rho,\phi)}\right)^k$ for any $k\geq 1$ as
$t \longrightarrow +\infty$.
\end{proposition}
\begin{proof}
Thanks to (\ref{taustar}),
we have that the copula between
$Y_{t}$ and $Y_{t+k}$ is gaussian with parameter
$$\tau_{t,t+k}=\prod_{s=0}^{k-1} \frac{\phi V_{t+s}+\rho \sigma_{\xi}}{V_{t+s+1}}.$$
Therefore, since for any $s\geq
1$ as $t \longrightarrow +\infty$
\begin{equation}\label{f1}\frac{V_{t+s}+\rho \sigma_{\xi}}{V_{t+s+1}}\longrightarrow
\frac{\bar{V}(\rho,\phi)+\rho \sigma_{\xi}}{\bar{V}(\rho,\phi)}=
\phi+\frac{\rho \sigma_{\xi}}{\bar{V}(\rho,\phi)},\end{equation}
we easily get the result.
\end{proof}
It is interesting to consider the role of $\rho$ in the dynamics
of the limit of $\tau_{t,t+k}$ with respect to the standard case
when $\tau_{t,t+k}=\phi^k$. If $\phi>0$ the limit autocorrelation
monotonically decreases to zero as $k$ increases, more slowly if
$\rho>0$ and more rapidly if $\rho<0$. On the contrary, if
$\phi<0$ the limit autocorrelation fluctuates more widely if
$\rho<0$ whereas the fluctuations are more blunt if $\rho>0$.

On the other hand, the innovations $(\xi_t)_t$ are of course no
longer serially independent as in the standard case and the $k$-th
order autocorrelation function approaches to a limit which again
depends on $\rho$, $\sigma_\xi$ and $k$.

\begin{proposition} Given the DIG-AR(1) Markov process of Definition \ref{DIG-AR}, the $k$-th order autocorrelation
function of $(\xi_t)_t$, $\delta_{t,t+k}$, for any $k\geq 1$, tends to
$$\delta_{t,t+k}\underset{t\uparrow +\infty}\longrightarrow\frac{\bar{V}(\rho,\phi)\bar{\tau}^{k-1}(\rho,\phi)}{\sigma_{\xi}^2}(\bar\tau(\rho,\phi)-\phi)(1-\phi\bar \tau(\rho,\phi))$$
where $\bar\tau$ is defined in (\ref{tau1}).
\end{proposition}
\begin{proof}
We have
$$\begin{aligned}\mathbb{E}[\xi_t
\xi_{t+k}]&=\mathbb{E}[(Y_t-\phi Y_{t-1})(Y_{t+k}-\phi Y_{t+k-1})]=\\
&=\mathbb{E}[Y_t Y_{t+k}]-\phi\mathbb{E}[Y_t Y_{t+k-1}]-\phi
\mathbb{E}[Y_{t-1} Y_{t+k}]+\phi^2 \mathbb{E}[Y_{t-1}
Y_{t+k-1}]=\\
&=\tau_{t,t+k}V_t V_{t+k}-\phi \tau_{t,t+k-1}V_t
V_{t+k-1}-\phi\tau_{t-1,t+k}V_{t-1}
V_{t+k}+\phi^2\tau_{t-1,t+k-1}V_{t-1} V_{t+k-1}.\end{aligned}$$
Since for any fixed $k\geq 1$, $\tau_{t,t+k} \longrightarrow
\bar{\tau}^{k-1}(\rho,\phi)$ and $V_t\longrightarrow
\bar{V}(\rho,\phi)$ as $t \longrightarrow +\infty$ we get
$$\begin{aligned}\mathbb{E}[\xi_t\xi_{t+k}]\longrightarrow & \bar{V}^2(\rho,\phi)\left(\bar{\tau}^{k}(\rho,\phi)
-\phi \bar{\tau}^{k+1}(\rho,\phi)-\phi
\bar{\tau}^{k-1}(\rho,\phi)+\phi^2\bar{\tau}^{k}(\rho,\phi)\right)=\\&
=\bar{V}^2(\rho,\phi)\bar{\tau}^{k-1}(\rho,\phi)(\bar\tau(\rho,\phi)-\phi)(1-\phi\bar \tau(\rho,\phi)).
\end{aligned}$$
It follows
$$\delta_{t,t+k}\longrightarrow \frac{\bar{V}^2(\rho,\phi)\bar{\tau}^{k-1}(\rho,\phi)}{\sigma^2_{\xi}}(\bar\tau(\rho,\phi)-\phi)(1-\phi\bar \tau(\rho,\phi)).$$
\end{proof}

To conclude this section we briefly discuss the mixing properties of the
sequence $(Y_t)_t$. Given a (not necessarily stationary) sequence
of random variables $(X_t)_{t \in \mathbb{Z}}$, defined on a given probability space $(\Omega,\mathcal F,\mathbb P)$, let
$\mathcal{F}_{t}^{l}$ be the $\sigma$-field
$\mathcal{F}_{t}^{l}=\sigma(X_t,t\leq t \leq l)$ with $-\infty
\leq t \leq l \leq +\infty$ and let
\begin{equation}\label{beta1}\tilde{\beta}(\mathcal{F}_{-\infty}^{t},\mathcal{F}_{t+k}^{+\infty})=\sup_{\{A_i\},\{B_j\}}\frac{1}{2}\sum_{i=1}^I \sum_{j=1}^J|\mathbb{
P}(A_i\cap B_j)-\mathbb{P}(A_i)\mathbb{P}(B_j)|,\end{equation}
where the
second supremum is taken over all finite partitions
$\{A_1,...A_I\}$ and $\{B_1,...B_J\}$ of $\Omega$ such that $A_i
\in \mathcal{F}_{-\infty}^t$ for each $i$ and $B_j \in
\mathcal{F}_{t+k}^{\infty}$ for each $j$. Define the following
dependence coefficient
$$\beta_k=\sup_{t \in \mathbb{Z}}\tilde{\beta}(\mathcal{F}_{-\infty}^{t},\mathcal{F}_{t+k}^{+\infty}).$$
We say that the sequence $(X_t)_{t \in \mathbb{Z}}$ is
$\beta-$mixing (or absolutely regular) if $\beta_k \rightarrow 0$
as $k \rightarrow +\infty$ (see Volkonskii and Rozanov, 1959 and 1961, for more details).

\begin{proposition}\label{mix}
The DIG-AR(1) Markov process $(Y_t)_t$ is $\beta$-mixing.
\end{proposition}
\begin{proof}
The result is an immediate consequence of Theorem 2.1 in Gobbi and Mulinacci (2017) and the proof is very close to that of Corollary 3.1 therein.
In fact, since $\vert \tau_{t,t+1}\vert <1$ for all $t$ and its limit, $\phi+\frac{\rho\sigma_\epsilon}{\bar V(\rho,\phi)}$, satisfies
$\left\vert \phi+\frac{\rho\sigma_\epsilon}{\bar V(\rho,\phi)}\right\vert <1$, we have that $\hat\eta=\underset{t}\sup \vert\tau_{t,t+1}\vert<1$.
Of course, exactly as shown in the proof of Corollary 3.1 of Gobbi and Mulinacci (2017) the density of the copula associated to the vector $(Y_t,Y_{t+1})$ is uniformly bounded
in $L^2([0,1])$ and Theorem 2.1 of Gobbi and Mulinacci (2017) applies, allowing to conclude that the process is $\beta$-mixing.
\end{proof}

\section{An infeasible estimator of the autoregressive parameter}\label{sec4}
This section is devoted to the analysis of the consistency of the ordinary
least squares (OLS) estimator, $\hat{\phi}_T$, of the
autoregressive coefficient $\phi$ in our DIG-AR(1) markovian
model. Unfortunately, unlike the standard case, the OLS estimator is no more consistent. Therefore, we propose an infeasible estimator which is by construction
consistent and that it is proved to be $\sqrt T$-asymptotically normal.
\medskip

We recall that given a random sample $(Y_1,....,Y_T)$
generated by $(Y_t)_t$, the OLS estimator is
$\hat{\phi}_T=\frac{\sum_{t=2}^{T}Y_t Y_{t-1}}{\sum_{t=2}^{T}
Y^2_{t-1}}$. In the standard AR(1) model, since $(Y_t)_t$ is a
stationary and ergodic process we have that $\hat{\phi}_T
\stackrel{a.s.}{\longrightarrow} \phi$ as $T \longrightarrow
+\infty$ (see, among others, Anderson, 1971 and Hamilton, 1994). This is no more the case in our more general context.

\begin{proposition}\label{pp1} Let $(Y_t)_t$ be a markovian DIG-AR(1) process as in Definition \ref{DIG-AR}. Then, the OLS estimator
$\hat{\phi}_T$ of $\phi$ is not consistent; in particular
$$\hat{\phi}_T
\stackrel{a.s.}{\longrightarrow} \bar\tau(\rho,\phi)$$
where $\bar \tau(\rho,\phi)$ is defined in (\ref{tau1}).
\end{proposition}
Let's observe that, coherently with the standard AR(1) process,
the OLS estimator converges to the limit of the first-order
autocorrelation. Unfortunately, in our DIG-AR(1) model such a
first-order autocorrelation does not coincide with the
autoregressive parameter.
\begin{proof} First we observe that the assumption of markovianity
of $(Y_t)_t$ implies that, if $(\mathcal F_t)_t$ is the filtration generated by $(Y_t)_t$,
then $\mathbb P(\xi_t\leq x\vert \mathcal F_{t-1})=\mathbb P(\xi_t\leq x\vert Y_{t-1})$ for all $x\in\mathbb R$ and, thanks to the assumptions of the model, the conditional distribution of $\xi_t$ given $Y_{t-1}$ is
\begin{equation}\label{conddist}(\xi_t|Y_{t-1})\sim N\left(\frac{\rho
\sigma_{\xi}}{V_{t-1}}Y_{t-1},\sigma_{\xi}^2(1-\rho^2)\right).\end{equation}
Now, consider the expression of $\hat{\phi}_T$,
$$\hat{\phi}_T=\frac{\sum_{t=2}^{T}Y_t Y_{t-1}}{\sum_{t=2}^{T}
Y^2_{t-1}}=\phi+\frac{\sum_{t=2}^{T}\xi_t Y_{t-1}}{\sum_{t=2}^{T}
Y^2_{t-1}}.$$ Differently from the standard case (see Andrews,
1988 and Hamilton, 1994) in our setting the ratio
$\frac{\sum_{t=2}^{T}\xi_t Y_{t-1}}{\sum_{t=2}^{T} Y^2_{t-1}}$
does not converge to zero as $T \longrightarrow +\infty$. In particular,
the process $(\xi_t Y_{t-1})_{t}$ is not a martingale difference
sequence and its mean is not constant over time,
$\mathbb{E}[\xi_t Y_{t-1}]=\rho \sigma_{\xi}V_{t-1}$. However, the
process
\begin{equation}\label{md}Z_t=\xi_t Y_{t-1}-Y_{t-1}\mathbb{E}[\xi_t |Y_{t-1}]\end{equation}
is a martingale difference
sequence with respect to the filtration $(\mathcal F_t)_t$ and, since thanks to equation (\ref{conddist}),
$Y_{t-1}\mathbb{E}[\xi_t |Y_{t-1}]=\rho
\sigma_{\xi}\frac{Y^2_{t-1}}{V_{t-1}}$, we have
\begin{equation}\label{OLS}\hat{\phi}_T-\phi=\frac{\sum_{t=2}^{T}Z_t}{\sum_{t=2}^{T}
Y^2_{t-1}}+\rho \sigma_{\xi}
\frac{\sum_{t=2}^{T}\frac{Y_{t-1}^2}{V_{t-1}}}{\sum_{t=2}^{T}
Y^2_{t-1}}.\end{equation}

\noindent Since, thanks to (\ref{conddist}),
\begin{equation}\label{sigmas}\mathbb{E}\left[Z_t^2\right]=\mathbb{E}\left[\mathbb{E}\left[Z_t^2|Y_{t-1}\right]\right]=\mathbb{E}\left[\sigma_{\xi}^2 Y_{t-1}^2
(1-\rho^2)\right]=\sigma_{\xi}^2
V^2_{t-1}(1-\rho^2)\end{equation}
and $V_{t}$ is a
convergent deterministic sequence and hence bounded by a positive constant $\Delta$, we have that
$$\sup_t\mathbb{E}\left[Z_t^2\right]=\sigma_{\xi}^2(1-\rho^2)\sup_{t}
V_{t-1}^2\leq \sigma_{\xi}^2(1-\rho^2)\Delta^2<+\infty.$$ Since
the sequence of the second moments of $ Z_t$ is bounded, we
can apply Theorem 3.77 in White (1984) and conclude that
$$\frac{1}{T}\sum_{t=2}^T Z_t
\stackrel{a.s.}{\longrightarrow} 0,$$ as $T \longrightarrow
+\infty$.

Notice that, since $(Y_t)_t$ is $\beta$-mixing (see Proposition \ref{mix}), also $(Y^2_{t})_t$ is $\beta$-mixing.\\
Moreover, (see Bradley, 2007) for a Markov
process $(X_t)_t$, (\ref{beta1}) can be rewritten as
$$\tilde{\beta}(\mathcal{F}_{-\infty}^{t},\mathcal{F}_{t,t+k}^{+\infty})= \frac{1}{2}\parallel
F_{t,t+k}(x,y)-F_t(x)F_{t+k}(y)\parallel_{TV}$$
where $\parallel\cdot\parallel _{TV}$ is the total variation norm in the Vitali sense and $F_{t,t+k}$, $F_{t}$ and $F_{t+k}$
are the cumulative distribution functions of $(X_t,X_{t+k})$, $X_{t}$ and $X_{t+k}$. If $(a_t)_t$ is a sequence of positive constants, it can be
easily shown that, as a consequence of the Vitali's total variation norm definition, the coefficient $\tilde{\beta}(\mathcal{F}_{-\infty}^{t},\mathcal{F}_{t,t+k}^{+\infty})$ calculated for the Markov process $\left (a_tX_t\right )_t$ is exactly
the same as that calculated for $(X_t)_t$.
This allows to conclude that  $\left(\frac{Y^2_{t-1}}{V_{t-1}}\right)_t$ is $\beta$-mixing.

Since $\sup_t\mathbb{E}\left[Y_{t-1}^4\right]=3\sup_t  V_{t-1}^{4}$ and
$\sup_t\mathbb{E}\left[\frac{Y_{t-1}^4}{V_{t-1}}\right]=3\sup_t
V_{t-1}^{2}$ are both finite, thanks to Theorem 3.57 in
White (1984), we have that
$$\frac{1}{T}\sum_{t=2}^T Y^2_{t-1}-\frac{1}{T}\sum_{t=2}^T V^2_{t-1}\stackrel{a.s.}{\longrightarrow} 0$$
and
$$\frac{1}{T}\sum_{t=2}^T \frac{Y^2_{t-1}}{V_{t-1}}-\frac{1}{T}\sum_{t=2}^T V_{t-1}\stackrel{a.s.}{\longrightarrow} 0.$$
But, thanks to Ces\`{a}ro averages theorem(see
appendix A30 in Billingsley, 1968), we have that
$\frac{1}{T}\sum_{t=2}^T V^2_{t-1} \longrightarrow
\bar{V}^2(\rho,\phi)$ and $\frac{1}{T}\sum_{t=2}^T
V_{t-1}\longrightarrow \bar{V}(\rho,\phi)$ and
\begin{equation}\label{ca}\frac{1}{T}\sum_{t=2}^T Y^2_{t-1}\stackrel{a.s.}{\longrightarrow}
\bar{V}^2(\rho,\phi)\text{ and }\frac{1}{T}\sum_{t=2}^T
\frac{Y^2_{t-1}}{V_{t-1}}\stackrel{a.s.}{\longrightarrow}
\bar{V}(\rho,\phi).\end{equation}
Substituting in (\ref{OLS}) we get the
conclusion.

\end{proof}
It is interesting to analyze the impact of the correlation coefficient $\rho$
on the asymptotic bias of the OLS estimator for different and
fixed values of the autoregressive parameter $\phi$. In
particular, Figure 2 shows the graph of the asymptotic bias
$\bar{\tau}(\rho,\phi_0)-\phi_0=\frac{\rho
\sigma_{\xi}}{\bar{V}(\rho,\phi_0)}$: we observe
that if $\rho$ and $\phi_0$ have the same sign the bias is low and
decreases if the value of the autoregressive parameter raises in
absolute value; on the other hand, if $\rho$ and $\phi_0$ have
opposite sign the bias is high and tends to increase with the
absolute value of the correlation.

\begin{figure}[h]
\begin{center}
\includegraphics[width=10cm,height=8cm]{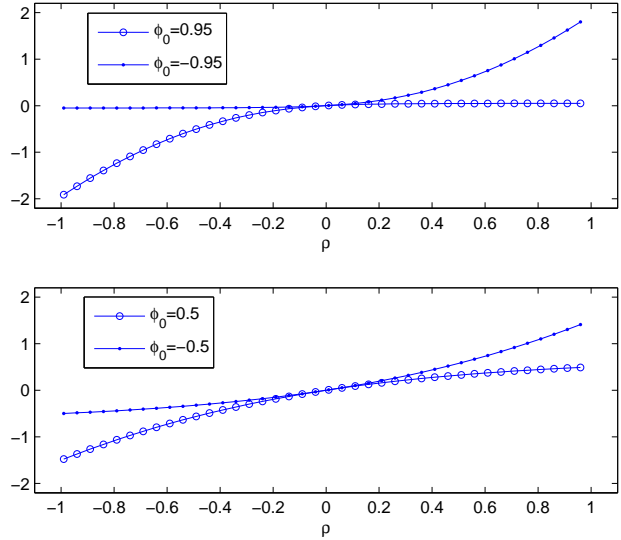}
\caption{Graph of the asymptotic bias
$\bar{\tau}(\rho,\phi_0)-\phi_0=\frac{\rho
\sigma_{\xi}}{\bar{V}(\rho,\phi_0)}$ as a function of $\rho$ for
different values of the autoregressive parameter $\phi_0$ Here we
fix $\sigma_{\xi}=1$.}\label{scatterCDB_RC}
\end{center}
\end{figure}

Thanks to Proposition \ref{pp1} we can introduce a new
infeasible estimator $\tilde\phi$ of $\phi$, which is consistent by construction, defined as
\begin{equation}\label{infeasibleestimator}
\tilde{\phi}_T=\hat{\phi}_T-\rho
\sigma_{\xi}\frac{\sum_{t=2}^T\frac{Y_{t-1}^2}{V_{t-1}}}{\sum_{t=2}^T
Y_{t-1}^2}
\end{equation}
where $\hat\phi$ is the OLS estimator of $\phi$.
\medskip

The following proposition provides a $\sqrt{T}$-asymptotic
normality for $\tilde{\phi}_T$.
\begin{theorem} Let $(Y_t)_t$ be a markovian DIG-AR(1) process as in Definition \ref{DIG-AR} and $\tilde{\phi}_T$ be the infeasible
estimator given in (\ref{infeasibleestimator}). Then
$$\sqrt{T}(\tilde{\phi}_T-\phi)\stackrel{L}{\longrightarrow} X\sim N(0,\bar{\eta}),$$
where
$\bar{\eta}=\frac{\sigma_{\xi}\sqrt{1-\rho^2}}{\bar{V}(\rho,\phi)}$ and $\stackrel{L}{\longrightarrow}$ denotes the convergence in distribution.
\end{theorem}
\begin{proof}
By (\ref{OLS}) we have
\begin{equation}\label{statement}\sqrt{T}(\tilde{\phi}_T-\phi)=\frac{\frac{1}{\sqrt{T}}{\sum_{t=2}^T Z_t}}{\frac{1}{T}\sum_{t=2}^T Y_{t-1}^2}.\end{equation}
where $(Z_t)_t$ is the martingale difference defined in (\ref{md}). Since, see (\ref{ca}), $\frac{1}{T}\sum_{t=2}^T
Y_{t-1}^2\stackrel{a.s.}{\longrightarrow} \bar{V}^2(\rho,\phi)$, in order to prove the statement, we
apply the central limit theorem of Proposition 7.8 in Hamilton (1994) to the numerator in (\ref{statement}).
Let us show that the assumptions of that Proposition are satisfied.

Let $\sigma_{t}^2=\mathbb{E}\left[Z_t^{2}\right]$: thanks to (\ref{sigmas}) and the assumptions of the model, $\sigma_{t}^2>0$ for all $t$
and, using Ces\`{a}ro averages theorem,
$$\frac{1}{T}\sum_{t=2}^T\sigma_t^{2} =\sigma_{\xi}^2
(1-\rho^2)\frac{1}{T}\sum_{t=2}^T V^2_{t-1}\longrightarrow
\sigma_{\xi}^2 (1-\rho^2)\bar{V}^2(\rho,\phi)=\bar{\sigma}^2,$$
which is strictly positive for all $|\rho|<1$ and all $|\phi|<1$. Moreover,
it is a straightforward computation to show that
$\sup_t\mathbb{E}\left[Z_t^4\right]<+\infty$ and, in order to apply the mentioned central limit theorem, it remains to prove that $\frac{1}{T}\sum_{t=2}^T
Z_t^{2}\stackrel{\mathbb{P}}{\longrightarrow}
\bar{\sigma}^2$.

Since
$\mathbb{E}\left[Z_t^{2}|Y_{t-1}\right]=\sigma_{\xi}^2Y_{t-1}^2(1-\rho^2)$,
the process $(W_t)_t$ defined as
$$W_t=Z_t^{2}-\sigma_{\xi}^2Y_{t-1}^2(1-\rho^2)$$
is a martingale difference sequence with respect to the filtration $(\mathcal F_t)_t$ generated by the process $(Y_t)_t$. Moreover, since it can be easily shown that
$\sup_t
\mathbb{E}\left[W_t^2\right]<+\infty$, we can apply Theorem 3.77 in White (1984) to the martingale difference $(W_t)_t$ and conclude that
$$\frac{1}{T}\sum_{t=2}^TW_t
\stackrel{a.s.}{\longrightarrow} 0.$$
Therefore
$$\frac{1}{T}\sum_{t=2}^T Z_t^{2}=
\frac{1}{T}\sum_{t=2}^TW_t+\sigma_{\xi}^2(1-\rho^2)\frac{1}{T}\sum_{t=2}^T
Y_{t-1}^2 \stackrel{a.s.}{\longrightarrow}\bar{\sigma}^2,$$ as
required.
Hence, by Proposition 7.8 in Hamilton (1994), we have that
$\frac{1}{\sqrt{T}}\sum_{t=2}^T Z_t$ converges in distribution to a random variable distributed as 
$N(0,\bar{\sigma})$
and, in conclusion,
$$\sqrt{T}(\tilde{\phi}_T-\phi)\stackrel{L}{\longrightarrow} X
\sim
N\left(0,\frac{\sigma_{\xi}\sqrt{1-\rho^2}}{\bar{V}(\rho,\phi)}\right),$$
which is the statement of the theorem.
\end{proof}

\section{Concluding remarks}\label{sec5}

The paper proposes a modified version of the standard AR(1)
process allowing for a time-invariant dependence between the state
variable $Y_{t-1}$ and the next innovation $\xi_t$. This induces
temporal dependence in the sequence of innovations and the
standard methods of derivation of the asymptotic properties of the
estimators of parameters are no longer applicable. In fact, we
show that the standard OLS estimator, generally used for the
estimation of the autoregressive parameter, is no more consistent
and its asymptotic bias is computed. As a consequence we propose a
new infeasible estimator and we establish its
$\sqrt{T}$-asymptotic normality.

\end{document}